\documentclass[12pt]{amsart}

\usepackage[a4paper,centering,margin=2.5cm]{geometry}
\usepackage{cite}
\usepackage{enumerate}
\usepackage{xcolor,hyperref,cleveref}
\definecolor{darkblue}{rgb}{0.0,0.0,0.3}
\hypersetup{colorlinks,breaklinks,
            linkcolor=darkblue,urlcolor=darkblue,
            anchorcolor=darkblue,citecolor=darkblue}

\usepackage{amssymb}

\usepackage{graphicx,psfrag}
\def\boxit{$\sqcap\kern-8pt\sqcup$}

\newcommand{\R}{\mathbb{R}}
\newcommand{\p}{\mathbb{P}}

\theoremstyle{plain}% default
\newtheorem{theorem}{Theorem}[section]

\newtheorem{proposition}[theorem]{Proposition}
\newtheorem{corollary}[theorem]{Corollary}

\newtheorem{remark}[theorem]{Remark}
\newtheorem{question}[theorem]{Question}

\theoremstyle{definition}

\newtheorem{conjecture}[theorem]{Conjecture}

\newtheorem{problem*}{Problem}

\usepackage{todonotes}

\title{Roudneff's Conjecture in Dimension $4$}

\dedicatory{}

\author[R. Hern\'andez-Ortiz]{Rangel Hern\'andez-Ortiz}
\address{Universitat Rovira i Virgili, Departament d'Enginyeria Inform\`{a}tica i Matem\`{a}tiques\\ Av. Pa\"{i}sos Catalans 26, 43007 Tarragona, Spain.}
\email{rangel.hernandez@urv.cat}

\author[K. Knauer]{Kolja Knauer}
\address{Departament de Matem\`atiques i Inform\`atica, Universitat de Barcelona, Spain \newline LIS, Aix-Marseille Universit\'e, CNRS, and Universit\'e de Toulon, Marseille, France}
\email{kolja.knauer@ub.edu}

\author[L. P. Montejano]{Luis Pedro Montejano}
\address{Serra H\'unter Fellow, Universitat Rovira i Virgili, Departament d'Enginyeria Inform\`{a}tica i Matem\`{a}tiques\\ Av. Pa\"{i}sos Catalans 26, 43007 Tarragona, Spain.}
\email{luispedro.montejano@urv.cat}

\author[M. Scheucher]{Manfred Scheucher}
\address{Institut f\"ur Mathematik, 
Technische Universit\"at Berlin, Germany}
\email{lastname@math.tu-berlin.de}

\date{\today}

\begin{document}

\maketitle

\begin{abstract}
J.-P. Roudneff conjectured in 1991 that every arrangement of $n \ge 2d+1\ge 5$ pseudohyperplanes in the real projective space $\mathbb{P}^d$ has at most $\sum_{i=0}^{d-2} \binom{n-1}{i}$ \emph{complete cells} (i.e., cells bounded by each hyperplane). The conjecture is true for $d=2,3$ and for arrangements arising from Lawrence oriented matroids.  The main result of this manuscript is to show the validity of Roudneff's conjecture for $d=4$. Moreover, based on computational data we conjecture that the maximum number of complete cells is only obtained by cyclic arrangements.
\end{abstract}

\section{Introduction}

A {\em projective} arrangement of $n$
pseudohyperplanes $H(d,n)$ in the real projective space $\p^d$ is a finite collection of %pseudohyperplanes, i.e., 
mildly deformed linear hyperplanes with several combinatorial properties, see Section~\ref{sec:TopRep} for the definition in terms of oriented matroids. In particular, no point belongs to every pseudohyperplane of $H(d,n)$.
Any
arrangement $H(d,n)$ decomposes  $\p^d$ into a
$d$-dimensional cell complex and any $d$-cell $c$ of
$H(d,n)$ has at most $n$ {\em facets} (that is, $(d-1)$-cells). We
say that a $d$-cell $c$ is a {\em complete cell} of $H(d,n)$ if $c$ has exactly $n$
facets, i.e., $c$ is  bounded by each pseudohyperplane of $H(d,n)$.

\medskip  %keep space

The \emph{cyclic polytope} of dimension $d$ with $n$ vertices,
%$C_d(t_1,\ldots, t_n)$,
discovered by Carath\'eodory \cite{Car}, is
the  convex hull in $\R^d$ of $n\ge d+1\ge3$ different points
$x(t_1),\dots ,x(t_n)$ on the moment curve $x: \R\to
\R^d, \ t \mapsto (t,t^2,\dots ,t^d)$.
Cyclic polytopes play an important role in combinatorial  convex geometry due to their
connection with certain extremal problems. 
See for example, the upper
bound theorem due to McMullen \cite{Mac}.  
\emph{Cyclic arrangements} are defined as the dual of the cyclic polytopes.
As for cyclic polytopes, cyclic arrangements also have extremal
properties, see Section~\ref{sec:TopRep} for the definition in terms of oriented matroids. For instance, Shannon \cite{Sh} introduced cyclic
arrangements as examples of projective
arrangements in dimension $d$ which minimize the number of cells with $(d+1)$ facets.

\medskip  %keep space

Denote by $C_d(n)$ the number of complete cells of
the cyclic arrangement of dimension $d$ with $n$ hyperplanes. Roudneff \cite{R91} proved that  $C_d(n)\ge \sum\limits_{i=0}^{d-2}
\binom{n-1}{i}$ holds for $d \ge 2$ and that this bound is tight for all
$n\ge 2d+1$. Moreover, he conjectured that in that case, cyclic arrangements
maximize the number of complete cells.

\begin{conjecture}[\!\!{\cite[Conjecture 2.2]{R91}}] \label{conj_Roudneff} 
Every arrangement of $n\ge 2d+1\ge 5$ pseudohyperplanes in~$\p^d$ has at most
$\sum\limits_{i=0}^{d-2}
\binom{n-1}{i}$ complete cells.
\end{conjecture}

The conjecture is true for $d=2$ (that is, any arrangement of $n$ pseudolines in~$\p^2$ contains at most one complete cell),
Ram\'irez Alfons\'in \cite{R99} proved the case $d=3$, and in \cite{MR15} the authors proved it for arrangements corresponding to Lawrence
oriented matroids.

\smallskip  %keep space

In \cite{FR01}  the exact
number of complete cells of cyclic arrangements was calculated for any positive
integers $d$ and $n$ with  $n\ge d+1$, namely, 
$$C_d(n)=\binom{d}{n-d}+
\binom{d-1}{n-d-1}+\sum_{i=0}^{d-2} \binom{n-1}{i}.$$
Thus, in view of Roudneff's conjecture,  the following question was asked in \cite{MR15}.

\begin{question}\label{question_smaller_n} Is it true that every arrangement of $n\ge d+1\ge 3$ pseudohyperplanes in~$\p^{d}$ has at most $C_d(n)$ complete cells?
\end{question}

Notice that there is a unique arrangement of 3 (resp.\ 4) 
lines in~$\p^2$ with
$C_2(3)=4$ (resp.\ $C_2(4)=3$) complete cells. Since Conjecture
\ref{conj_Roudneff} is true for $d=2$ and $n\ge 5$,
Question~\ref{question_smaller_n} is answered affirmatively for $d=2$.

\smallskip  %keep space

As the main result of this paper, 
we give an affirmative answer to Question~\ref{question_smaller_n} for $d=4$
and therefore prove Roudneff's conjecture for dimension $4$, further supporting the general conjecture. In addition, with a few simple observations, we answer Question~\ref{question_smaller_n} for $d=3$ and further strengthen Roudneff's conjecture.

\section{Oriented matroids}

Let us give some basic notions and definitions in oriented matroid theory. We assume some knowledge and standard notation of the theory of oriented matroids, for further reference the reader can consult the textbook  \cite{BVSWZ99}. 
A \emph{signed set} or \emph{signed vector} $X$ on ground set $E$ is a set $\underline{X}\subseteq E$ together with a partition $(X^+, X^-)$ of $\underline{X}$ into two distinguished subsets: $X^+$, the set of \emph{positive elements} of $X$, and $X^-$, its set of
\emph{negative elements}. The set $\underline{X}=X^+\cup X^-$ is the \emph{support} of $X$. We denote by $-X$ the sign-vector such that $-X^+=X^-$ and $-X^-=X^+$.  % i.e, $\underline{X}=E\setminus X^0$
An oriented matroid $\mathcal{M}=(E,\mathcal{C})$ is a pair of a finite ground set $E$ and a collection of signed sets on $E$ called \emph{circuits}, satisfying the following axioms:
\begin{itemize}
\item $\emptyset\notin\mathcal{C}$,
\item if $X\in \mathcal{C}$ then $-X\in \mathcal{C}$,
\item if $X,Y\in \mathcal{C}$ and  $\underline{X}\subseteq \underline{Y}$ then $X=\pm Y$, 
\item if $X,Y\in \mathcal{C}$, $X\neq- Y$, and $e\in X^+\cap Y^-$, then there is $Z\in \mathcal{C}$, with $e\notin\underline{Z}$ and $Z^+\subseteq X^+\cap Y^+$ and $Z^-\subseteq X^-\cap Y^-$.
\end{itemize}

We say that $X\in \mathcal{C}$ is a \emph{positive circuit} if $X^-=\emptyset$. We call the set of all reorientations of $\mathcal{M}$  its \emph{reorientation class}. We say that $\mathcal{M}$ is \emph{acyclic} if it does not contain positive circuits (otherwise, $\mathcal{M}$ is called \emph{cyclic}). 
A {\em reorientation} of $\mathcal{M}$ on $R\subseteq E$ is performed by changing the signs of the elements in $R$ in all the circuits of~$\mathcal{M}$. It is easy to check that the new set of signed circuits is also the set of circuits of an oriented matroid, usually denoted by $\mathcal{M}_R$. A reorientation is {\em acyclic} if $\mathcal{M}_R$ is acyclic.
Recall that oriented matroid on $n$ elements is {\em uniform of rank $r$ } if the set of supports of its circuits consists of all $(r+1)$-element subsets  of $E$.  
Given a uniform oriented matroid $\mathcal{M}$ of rank $r$ on $n=|E|$ elements, we denote its dual by $\mathcal{M}^*$, which is another uniform oriented matroid of rank $n-r$ on $n$ elements.
\smallskip

A characterization of oriented matroids in terms of basis orientations (that we will not make explicit here) was given by Lawrence \cite{L82}.  Let $r\ge 1$ be an integer and $E=\{1,\ldots, n\}$ be a set. A mapping $\chi: E^r \rightarrow \{-1,0,1\}$ (where we will abbreviate it by $\{-,0,+\}$) is a basis orientation of an oriented matroid of rank $r$ on $E$ if and only if $\chi$ is a \emph{chirotope}, that is, a special alternating mapping not identically zero. It is known that $\chi: E^r \rightarrow \{-,+\}$ is a chirotope if and only if $\chi$ is a basis orientation of a rank $r$ uniform oriented matroid on~$E$. 
Moreover, if $\chi(B)=+$ for any ordered basis $B = (b_1,\ldots,b_r)$ of $\mathcal{M}$ with $b_1<\ldots<b_r$, then the uniform matroid $\mathcal{M}$ is known to be the \emph{alternating oriented matroid} of rank $r$ on $n$ elements. In that case, the signs of each  circuit alternate along the ordering of $E$.

\smallskip  %keep space

Given two sign-vectors $X,Y\in \{+,-,0\}^E$,  their \emph{separation} is the set $S(X,Y)=\{e\in E\mid X_e\cdot Y_e=-\}$, where $X_e$ and $Y_e$ are the sings of the element $e$ in  $\underline{X}$ and $\underline{Y}$, respectively. We denote by $X\perp Y$ and say that $X$ and $Y$ are \emph{orthogonal} if the sets $S(X,Y)$ and $S(X,-Y)$ are either both empty or both non-empty. % see~\cite{BVSWZ99}. %Page 5 book Oriented Matroids
Maximal covectors of an oriented matroid $\mathcal{M}$  are usually called \emph{topes}.
It is known that  a sign-vector  $T\in \{+,-\}^E$ is a tope of $\mathcal{M}$ if and only if   $T\perp X$ for all circuit  $X\in \mathcal{C}$ (see \cite[Section~1.2, page~14]{BVSWZ99}). % page 14, Orientes matroids, Las Vergnas, 2do inciso
 Moreover, $T$ is a tope of $\mathcal{M}$ if and only if  $S(T,X)$ and $S(T,-X)$ are both non-empty, for every circuit $X\in \mathcal{C}$.
 % page 14, Orientes matroids, Las Vergnas, 1er inciso.
 %The reason is the following: since $|\underline{X}|>0$.
 % i.e., $X^+\cup X^- >0$
 %i.e, The case "either both empty" is not possible

\subsection{Topological Representation Theorem}\label{sec:TopRep}

The combinatorial properties of arrangements of pseudohyperplanes can be studied in the language of oriented matroids. The Folkman--Lawrence topological representation theorem~\cite{FL78} states that the  reorientation classes of oriented matroids on $n$ elements and rank $r$ (without loops or parallel elements) are in one-to-one correspondence with the classes of isomorphism of arrangements of $n$ pseudospheres in~$S^{r-1}$ (see  \cite[Theorem 1.4.1]{BVSWZ99}). There is a natural identification between pseudospheres and pseudohyperplanes as follows. Recall that $\p^{r-1}$ is the topological space obtained from $S^d$ by identifying all pairs of antipodal points. The double covering map $\pi: S^{r-1} \rightarrow \p^{r-1}$, given by $\pi(x)=\{x,-x\}$, gives an identification of centrally symmetric subsets of $S^{r-1}$ and general subsets of $\p^{r-1}$. This way centrally symmetric pseudospheres in $S^{r-1}$ correspond to pseudohyperplanes in $\p^{r-1}$. Hence, the topological representation theorem can also be stated in terms of pseudohyperplanes in $\p^{r-1}$, i.e., the reorientation classes of  oriented matroids on $n$ elements and rank $r$ (without loops or parallel elements) are in one-to-one correspondence with the classes of isomorphism of arrangements of $n$ pseudohyperplanes in~$\p^{r-1}$ (see \cite[Section 5, exercise 5.8]{BVSWZ99}). 

\smallskip  %keep space

An arrangement $H(d,n)$ is called {\em simple} if every intersection of $d$ pseudohyperplanes is a unique distinct point.  Simple arrangements correspond to uniform oriented matroids.  The $d$-cells of any arrangement $H(d,n)$ are usually called topes since they are in one-to-one correspondence with the topes of each of the  oriented matroids $\mathcal{M}$ of rank $r=d+1$ on $n$ elements of its corresponding reorientation class. It is known %from the Topological Representation Theorem
that a tope of $\mathcal{M}$ (i.e, a $d$-cell of its corresponding arrangement) corresponds to an acyclic reorientation of $\mathcal{M}$ having as interior elements precisely those pseudohyperplanes not bordering the tope. Moreover,  a tope $T$ of $\mathcal{M}$ is a complete cell if reorienting any single element of $T$, the resulting sign-vector is also a tope of $\mathcal{M}$. 
Cyclic arrangements of $n$ hyperplanes in~$\p^{d}$ are equivalent to
alternating oriented matroids of rank $r=d+1$ on $n$ elements, which hence have exactly $2C_{r-1}(n)$ complete cells. Summarizing,  Question~\ref{question_smaller_n} (and hence Roudneff's conjecture) can be stated in the following form:

\begin{quote}
 \emph{Every rank $r$ oriented matroid $\mathcal{M}$ on $n\geq r+1$ elements has at most $2C_{r-1}(n)$ complete cells.}
\end{quote}

We summarize for later usage:
Given a rank $r$ oriented matroid $\mathcal{M}=(E,\mathcal{C})$, the following three conditions hold.

\begin{enumerate}[(a)]
  \item
  A tope of $\mathcal{M}$ is a sign-vector  $T\in \{+,-\}^E$  such that  $T\perp X$  for all circuit $X\in \mathcal{C}$. 
  \label{condition:topes_def}
  
  \item 
  A tope $T$ of $\mathcal{M}$ is a complete cell if reorienting any single element of $T$, the resulting sign-vector is also a tope of $\mathcal{M}$. \label{condition:tope=complet_cell}
  
  \item
  If the corresponding arrangement of $n$ pseudohyperplanes in~$\p^{r-1}$ of $\mathcal{M}$ is simple, then $\mathcal{M}$ is uniform. \label{condition:simple=uniform}
\end{enumerate}

\section{Previous results}

We will use the following result due to Roudneff.

\begin{proposition}[\!\!\cite{R91}]\label{prop_Roudneff}
  To prove Conjecture~\ref{conj_Roudneff} for dimension $d$, it suffices to verify it for all simple arrangements of $n=2d+1$ pseudohyperplanes in  $\p^d$.
\end{proposition}

From the proof of the above proposition, it can be seen that even for any arrangement $H$ with $n \le 2d$ pseudohyperplanes in  $\p^d$, we may also perturb each hyperplane of $H$ a bit in order to obtain  a simple arrangement $H'$ with at least the same number of complete cells as $H$ (see Proposition 2.3 of \cite{R91}). This shows that also for Question~\ref{question_smaller_n}, we can restrict ourselves to simple arrangements.

\begin{remark}\label{simple_arrang}
  To answer Question~\ref{question_smaller_n}  for dimension $d$ in the affirmative, it suffices to verify it for simple arrangements of  pseudohyperplanes in  $\p^d$.
\end{remark}

Thus, by condition~(\ref{condition:simple=uniform}) 
and by Remark~\ref{simple_arrang}, it is sufficient to prove Question~\ref{question_smaller_n} for uniform oriented matroids. The following observation will be useful in this work.

\begin{remark}\label{oneclass_r+2}
There is only one reorientation class of uniform rank $r$  oriented matroids on $n\le r+2$ elements.
\end{remark}
\begin{proof}
The number of reorientation classes of a uniform oriented matroid $\mathcal{M}$ of rank $r$ on $n$ elements is equal to the number of reorientation classes of its dual  $\mathcal{M}^*$.
Now, if  $\mathcal{M}$ has rank $r$  and $n\le r+2$ elements, then $\mathcal{M}^*$ has rank at most $2$. Hence, $\mathcal{M}^*$ and therefore $\mathcal{M}$ has only one reorientation class. 
\end{proof}

Thus, every acyclic uniform oriented matroid on at most $r+2$ elements is in the reorientation class of the alternating oriented matroid and hence, they all have the same number of complete cells.
As a consequence of Remarks \ref{simple_arrang} and \ref{oneclass_r+2}, we can answer affirmatively Question~\ref{question_smaller_n} for $n\le r+2$.  In particular, as for $r=4$ (dimension $d=3$) Conjecture~\ref{conj_Roudneff} is true for $n\ge 7$, we obtain the following.

\begin{corollary}
Every arrangement of $n\ge 4$ pseudohyperplanes in~$\p^{3}$ has at most $C_3(n)$ complete cells.
\end{corollary}

\section{Main result}

Given a uniform rank $r$ oriented matroid $\mathcal{M}=(E,\mathcal{C})$ on $n=|E|$ elements, we explain the procedure to obtain the set of all complete cells of its corresponding arrangement of $n$ pseudohyperplanes in  $\p^d$ via the signed bases of $\mathcal{M}$. We start with the signature of all the bases of $\mathcal{M}$ and then, we obtain all its signed circuits. After that, we get the set of topes of $\mathcal{M}$ and finally, we obtain  the set of all complete cells of $\mathcal{M}$ as follows:

\medskip  %keep space

\emph{Bases $\rightarrow$ Circuits:}
From the chirotope, we may obtain  that 
$$\chi(B)=-X_{b_i}\cdot X_{b_{i+1}}\cdot \chi(B'),$$ 
where $\underline{X}=\{b_1,...,b_{r+1}\}$ is the support of an ordered circuit of $\mathcal{M}$ and $B=\underline{X} - {b_i}$ and $B'=\underline{X} - {b_{i+1}}$ are two bases of $\mathcal{M}$ (see \cite[Section 3.5]{BVSWZ99}). Hence, given $\chi(B)$, for any basis $B$ of $\mathcal{M}$,  we obtain the signed circuit $X$ and since $\mathcal{M}$ is uniform, we can proceed to obtain all the signed circuits of   $\mathcal{M}$.

\medskip %keep space

\emph{Circuits $\rightarrow$ Topes:}
For any sign-vector $T\in \{+,-\}^n$,  we verify condition~(\ref{condition:topes_def}) to confirm that $T$ is a tope of $\mathcal{M}$, i.e.,  we check for all circuit $X\in \mathcal{C}$ of $\mathcal{M}$
if $S(T,X)$ and $S(T,-X)$ are both non-empty (see \cite[Section~1.2, page~14]{BVSWZ99}).

\medskip %keep space

\emph{Topes $\rightarrow$ Complete cells:}
For any tope $T$,  we verify condition~(\ref{condition:tope=complet_cell})  to  confirm that $T$ is a complete cell of $\mathcal{M}$. That is, we reorient any single element of $T$, check if the resulting sign-vector is a tope of $\mathcal{M}$ and verify this for each of the $n$ entries of $T$.

\medskip %keep space

Finschi and Fukuda \cite{FinschiFukuda2002,Finschi2001}
enumerated the signed bases of all the reorientation classes of uniform rank~$5$ oriented matroids on~$8$ and $9$~elements.
While the data for $8$ elements is available on the website~\cite{FinschiDBOM}, 
the data for $9$ elements and also their source code for the enumeration is available upon request from Lukas Finschi.
We follow the procedure explained above with a computer program (available at \cite{supplemental_data}) which gives us the number of complete cells of each acyclic reorientation class.
After about 26 CPU days of computing time 
(i.e.,\ few days with parallelization), 
we obtain the following.

\begin{theorem}\label{r5n8}
Each of the $135$ 
reorientation classes of uniform rank $5$ oriented matroids on $8$ elements has at most $2C_4(8)$ complete cells. Moreover, the class 
of the alternating oriented matroid is the only one with exactly $2C_4(8)$ complete cells.
\end{theorem}

\begin{theorem}\label{r5n9}
Each of the $9276595$  reorientation classes of uniform rank $5$ oriented matroids on $9$ elements has at most $2C_4(9)$ complete cells. Moreover, the class 
of the alternating oriented matroid is the only one with exactly $2C_4(9)$ complete cells.
\end{theorem}

We can now prove our main result:

\begin{theorem}
Every arrangement of $n\ge  5$ pseudohyperplanes in~$\p^4$ has at most
$C_4(n)$ complete cells.
\end{theorem}
\begin{proof}
By Proposition~\ref{simple_arrang}, it is sufficient to prove the theorem for simple arrangements, that is, for uniform oriented matroids (see condition~(\ref{condition:simple=uniform})). Thus, by Remark~\ref{oneclass_r+2} and Theorem~\ref{r5n8}, the result holds for $n=5,6,7$ and~$8$. Finally, by Proposition~\ref{prop_Roudneff}  it suffices to verify it for $n=9$. Therefore,  the result holds by Theorem~\ref{r5n9}.
\end{proof}

Finally, we have used our computer program to verify that the cyclic arrangement is the unique example which maximizes the number of complete cells for $d=2$ and $n\le 10$, for $d=3$ and $n\le7$, and for $d=4$ and $n\le 9$.
Based on our computational evidence, we conclude this article 
with the following strengthening of Roudneff's conjecture and Question~\ref{question_smaller_n}: 
\begin{conjecture}
\label{conj_strengthening}
Every arrangement of $n\ge d+1\ge 3$ pseudohyperplanes in~$\p^d$ has at most
$C_d(n)$ complete cells.
Moreover, 
among all arrangements of $n$ pseudohyperplanes in~$\p^{d}$ the cyclic arrangement is (up to isomorphism) the only one with $C_d(n)$ complete cells.
\end{conjecture}

Last but not least, as the proof of Proposition~\ref{prop_Roudneff},
it suffices to verify Conjecture~\ref{conj_strengthening} for simple arrangements of pseudohyperplanes in~$\p^d$.
However, we do not know whether the setting can also be restricted to $ n \le 2d+1$ without loss of generality.

\subsection*{Acknowledgements.}
L.\ P.\ M.\  was supported by SGR Grant 2021 00115.

M.\ S.\ was supported by DFG Grant SCHE~2214/1-1.

K.\ K.\ was supported by the Spanish \emph{Ministerio de Econom\'ia,
Industria y Competitividad} through grants RYC-2017-22701 and ALCOIN: PID2019-104844GB-I00.

\end{document}